\documentclass{amsart}
\usepackage{amssymb}
\input xy
\xyoption{all}

\newcommand{\B}{\mathbb{B}}
\newcommand{\C}{\mathbb{C}}
\newcommand{\D}{\mathbb{D}}

\newcommand{\N}{\mathbb{N}}

\newcommand{\Z}{\mathbb{Z}}
\renewcommand{\P}{\mathbb{P}}
\newcommand{\R}{\mathbb{R}}

\newcommand{\bT}{\mathbb{T}}

\renewcommand{\d}{\mathrm{d}}

\newcommand\E{\mathrm{e}}
\newcommand\I{\mathrm{i}}

\newcommand{\cA}{\mathcal{A}}

\newcommand{\cC}{\mathcal{C}}

\newcommand{\cL}{\mathcal{L}}

\newcommand{\cO}{\mathcal{O}}

\newcommand{\cP}{\mathcal{P}}

\newcommand\wt{\widetilde}

\newcommand{\Psh}{\mathrm{Psh}}
\newcommand{\supp}{\mathrm{supp}\,}

\newcommand{\ol}{\overline}

\newcommand\wh{\widehat}

\newcommand\bs{\backslash}

\newcommand\di{\partial}

\newcommand\dist{\mathrm{dist}}
\newcommand\dibar{\bar\partial}

\newtheorem{theorem}{Theorem}[section]

\newtheorem{corollary}[theorem]{Corollary}
\newtheorem{proposition}[theorem]{Proposition}

\theoremstyle{definition}
\newtheorem{definition}[theorem]{Definition}

\newtheorem{problem}[theorem]{Problem}
\newtheorem{remark}[theorem]{Remark}

\numberwithin{equation}{section}

\begin{document}
\title[Projective Hulls and Analytic Discs]  
{Characterizations of Projective Hulls by Analytic Discs}        
\author{Barbara Drinovec Drnov\v sek \& Franc Forstneri\v c}
\address{Faculty of Mathematics and Physics, University of Ljubljana, 
and Institute of Mathematics, Physics and Mechanics, Jadranska 19, 
1000 Ljubljana, Slovenia}
\email{barbara.drinovec@fmf.uni-lj.si}
\email{franc.forstneric@fmf.uni-lj.si}
\thanks{Research supported by grants P1-0291 and J1-2152, Republic of Slovenia.}

%
%
%    General information
%
%

\date{\today} 

%
%
%  ABSTRACT
%
%
\begin{abstract}
The notion of the projective hull of a compact set in a
complex projective space $\P^n$ was introduced by
Harvey and Lawson in 2006. 
In this paper we describe the projective hull 
by Poletsky sequences of analytic discs, in 
analogy to the known descriptions of the holomorphic
and the plurisubharmonic hull.
\end{abstract}

\maketitle

{\small \rightline{\em  Dedicated to John P.\ D'Angelo on the occasion of his 60th birthday}}

%%%%%%%%%%%%%%%%%%%%%%%%%%%%%%%%%%%%%%%%%%%%%%%%%%%%%%%%%%%%%%%%%%%%%
%                                                   								%
%																								    								%
%  PROJECTIVE HULLS                                              		%
% 										     																					%
%																								    								%
%																								    								%
%%%%%%%%%%%%%%%%%%%%%%%%%%%%%%%%%%%%%%%%%%%%%%%%%%%%%%%%%%%%%%%%%%%%%

\section{Projective hulls}
\label{sec:projective}
Given a compact set $K$ in the complex projective space $\P^n$, 
Harvey and Lawson \cite{Harvey-Lawson2006} 
introduced its {\em projective hull} $\wh K_{\P^n}$ 
as the set of all points $x\in \P^n$ for which there exists 
a constant $C(x) <+\infty$ satisfying
\begin{equation}
\label{eq:Cx}
    |\cP(x)| \le C(x)^d \sup_K |\cP|
\end{equation}
for all holomorphic sections $\cP$ of the line bundle 
$\cO_{\P^n}(d)$ and all integers $d>0$. 
Their principal motivation was to extend 
the notion of the polynomial hull and to generalize 
Wermer's classical theorem \cite{Wermer} 
that the polynomial hull of a closed real analytic curve 
$\gamma\subset \C^n$ is either $\gamma$ itself, 
or a complex curve with boundary $\gamma$.

In this note we give several descriptions  
of projective hulls by sequences of analytic discs, 
in analogy to Poletsky's description of the polynomial hull
\cite{Poletsky1993}. We also characterize
the polynomial hull of a compact connected 
circular set $K$ in $\C^n$ by analytic discs 
that have the entire boundary circle close to $K$ 
(see Theorem \ref{thm:hull-circular}.)

We begin by  recalling some of the main results from the 
paper \cite{Harvey-Lawson2006} by Harvey and Lawson,
with emphasis on those that are used in this paper.
In the sequel we let $C_K(x)\ge 1$ be the smallest constant 
satisfying (\ref{eq:Cx}) and set $C_K(x)=+\infty$ for $x\in\P^n\setminus \wh K_{\P^n}$;
this gives the {\em best constant function} $C_K\colon \P^n \to [1,+\infty]$.
 
There is a simple description of the projective hull 
in terms of the polynomial hull.
Let $\pi\colon \C^{n+1}_*=\C^{n+1}\setminus \{0\}\to \P^n$ be the 
standard projection taking a point $z=(z_0,\ldots,z_n)\in \C^{n+1}_*$
onto the point $x=\pi(z)\in \P^n$ with homogeneous coordinates 
$[z_0:\cdots:z_n]$. Denote by $L_x=\pi^{-1}(x)\cup\{0\}$ 
the complex line in $\C^{n+1}$ over $x$.
Let $\B=\{z\in \C^{n+1}\colon |z|<1\}$ 
and $S=b\B=\{z\in \C^{n+1}\colon |z|= 1\}$.
Given a compact set $K\subset \P^n$, we define
compact subsets $S_K\subset B_K$ of $\C^{n+1}$ by
\begin{equation}
\label{eq:BK}
	B_K=\ol \B \cap (\pi^{-1}(K)\cup \{0\}),\quad
	S_K=S\cap \pi^{-1}(K).
\end{equation}
Clearly these two sets have the same polynomial hull in $\C^{n+1}$,
$\wh B_K=\wh S_K$. The intersection of $\wh B_K$ with the complex
line $L_x$ over any point $x\in \P^n$ is a closed disc 
$\triangle_x$ in $L_x$ of radius $r(x)\ge 0$ centered at $0$.
Then we have (see \cite[Proposition 5.2]{Harvey-Lawson2006}) 
\begin{equation}
\label{eq:hull2}
	 \wh K_{\P^n}= \{x\in \P^n\colon r(x)>0\} 
	 \quad \text{and} \quad 
	 C_K(x) = \frac{1}{r(x)},
\end{equation}
where $C_K$ is the best constant function from (\ref{eq:Cx}).
Equivalently: % \cite[Corollary 5.3]{Harvey-Lawson2006}:
\begin{equation}
\label{eq:SK}
	\wh K_{\P^n}=\pi\left(\wh S_K \setminus \{0\}\right).
\end{equation}
This follows by observing that sections of $\cO_{\P^n}(d)$ 
correspond to homogeneous polynomials of degree $d$
on $\C^{n+1}$. Indeed, let us denote by $L=\cO_{\P^n}(-1)$
the universal line bundle over $\P^n$. The total space of $L$
is $\C^{n+1}$ blown up at the origin, with the zero section
$L_0\cong\P^{n}$ corresponding to the exceptional fiber,
and $L\setminus L_0\cong \C^{n+1}_*$. A homogeneous polynomial
of degree $d$ on $\C^{n+1}$ defines a linear functional on the $d$-th tensor
power $L^{\otimes d} \cong \cO_{\P^n}(-d)$ of $L$,  
and hence a holomorphic section of the dual bundle 
$(L^{\otimes d})^* \cong \cO_{\P^n}(d)$.
This shows that the projective hull equals the set
(\ref{eq:SK}) if we replace the polynomial hull of $S_K$ by the
(ostensibly larger) hull obtained by homogeneous polynomials.
Since $S_K$ is circular, these two hulls coincide 
\cite[Proposition 5.4]{Harvey-Lawson2006}.

Another characterization of the projective hull is given 
in terms of the extremal function of $K$ with respect to the family
$\Psh_\omega(\P^n)$ of all upper semicontinuous functions
$v\colon \P^n\to \R\cup\{-\infty\}$ satisfying 
%
%opomba referee zeli, da napiseva  \d\d^c=\di\dibar; tukaj se lahko pojavi konstanta
%ki je odvisna od tega, katero definicijo d^c vzames...
%
$\d\d^c v + \omega= 2\I \di\dibar v + \omega\ge 0$,
where $\omega=\omega_{FS}$ is the Fubini-Study form on $\P^n$. 
(Here $\d^c=\I(\dibar -\di)$.) Setting 
\[
	V^\omega_K = \sup\{ v\in \Psh_\omega(\P^n)\colon v\le 0 \ \text{on}\ K\},
\]
we have that
\begin{equation}
\label{eq:VK}
	\wh K_{\P^n}= \{x\in \P^n\colon V^\omega_K(x) < +\infty\},
	\quad V_K^\omega(x)=\log C_K(x). 
\end{equation}
In particular, $\wh K_{\P^n} \ne \P^n$ if and only if $K$ is $\omega$-pluripolar,
i.e., there exists $v\in \Psh_\omega(\P^n)$ with $v\not\equiv -\infty$
and $K\subset \{ v=-\infty \}$.

%opomba referee dodan prvi stavek in definicija Siciak-Zaharyuta extremal function za vse  mn.
%(glede na zadnjo ref. pripombo) definirana za vse množice (omejenosti Klimek ne predpostavi)
%
The {\em Lelong class} $\cL_{\C^n}$ on $\C^n$ 
is the set of all plurisubharmonic functions $v\colon \C^n\to\R\cup\{-\infty\}$
for which there exist constants $r>0$ and $C\in \R$ 
(depending on $v$) such that
\[
	v(z) \le \log |z| + C, \qquad z\in \C^n, \ |z|>r.
\]
Given a subset $E\subset \C^n$, the
{\em Siciak-Zaharyuta extremal function} $V_E\colon\C^n\to \R\cup\{\infty\}$ (see \cite{Klimek})
is defined by 
\[
	V_E(z) = \sup\{v(z)\colon v\in\cL,\ v|_E\le 0\}.
\] 
If $K$ is a compact set in an affine chart
$\C^n \subset \P^n$, then a point $z \in \C^n$ belongs to 
$\wh K_{\P^n}$ if and only if $V_K(z)<+\infty$. 
This is seen by considering $\C^n$ as the hyperplane 
$\{z_0=1\} \subset \C^{n+1}$ and observing that a polynomial 
of degree $d$ on $\C^n$ corresponds 
to a homogeneous polynomial of degree $d$ on $\C^{n+1}$.
One also has a bijective correspondence 
between the Lelong class $\cL_{\C^n}$ and $\Psh_\omega(\P^n)$
(see \cite{Bedford-Taylor-1988} and \cite[Example 1.2]{Guedj-Zeriahi-2005}).
Explicitly, given $v\in \cL_{\C^n}$, the function  
\[
	\wt v(z_0,\ldots,z_n)= \log|z_0| + 
	v\left(\frac{z_1}{z_0},\ldots, \frac{z_1}{z_0}\right)
\]
is plurisubharmonic on $\C^{n+1}_*$; its restriction
to the unit sphere $S$ is circle invariant and hence 
defines a function $v'$ on $\P^n$. It is easily seen that
$v\mapsto v'$ is a bijective map of $\cL_{\C^n}$ onto
$\Psh_\omega(\P^n)$. 

Since the polynomial hull $\wh K:=\wh K_{\cO(\C^n)}$ of a compact set
$K\subset\C^n\subset\P^n$ equals 
$\{V_K=0\}$, it is contained in the projective hull.
Conversely, if $\wh K_{\P^n}$ lies in the complement 
$\Omega= \P^n\setminus \Lambda$ of
an algebraic hypersurface $\Lambda\subset \P^n$, then
$\wh K_{\P^n} = \wh K_{\cO(\Omega)}$
\cite[Corollary 12.7]{Harvey-Lawson2006}.
However, it is in general impossible to describe
the projective hull of a compact affine set $K\subset \P^n$
in terms of its polynomial hulls in affine subsets of $\P^n$
containing $K$. For example, there exists a smooth closed curve in
$\C^2$ whose polynomial hull is a holomorphic disc bounded by
the curve, but whose projective hull is $\P^2$
\cite[Remark 4.5]{Harvey-Lawson2006}.

We now describe our main results.
Let $\D=\{\zeta\in \C\colon |\zeta|< 1\}$ be the open unit disc 
in $\C$ and let $\bT=b\D=\{\zeta\in\C\colon |\zeta|=1\}$ be
its boundary circle. An {\em analytic disc} in a complex
space $X$ is a continuous map $f\colon\ol\D\to X$
which is holomorphic in $\D$. The point $f(0)$ is called
the {\em center} of the disc. We denote the space of all such
discs by $\cA_X$.

Our first characterization of the projective hull,
which applies to any compact set in $\P^n$, is in 
terms of Poletsky sequences of  discs in $\P^n$ 
that have a bounded lifting property with respect to the
projection $\pi\colon\C^{n+1}\setminus \{0\} \to\P^n$
(see Theorem \ref{thm:main}). By a {\em Poletsky sequence 
of discs}  $f_j$ (for a given compact set $K$) we mean that 
the set of points $\E^{\I t}\in \bT$ for which 
$\dist(f_j(\E^{\I t}),K)<1/j$ has measure $>2\pi - 1/j$. 
The proof uses Poletsky's theorem characterizing 
polynomial hulls by analytic discs.

The second characterization is motivated by a result of 
Lawson and Wermer \cite{Lawson-Wermer} 
in the case when $K$ is a simple closed curve; we simplify their 
proof and extend it to any compact connected set 
$K$ contained in an affine chart $\C^n=\P^n\setminus H$ of $\P^n$.
The projective hull of $K$ is the set of all centers 
of sequences of analytic discs $f_j$ in $\P^n$, with boundary circles 
$f_j(\bT)$ converging to $K$, such that a certain disc functional
$J$ is uniformly bounded on the sequence  
(see Theorem \ref{thm:LW}). The functional $J(f)$,
which was first introduced by L\'arusson and Sigurdsson in
\cite{Larusson-Sigurdsson2005}, is determined by 
the intersection divisor of the disc with the hyperplane at infinity;
see (\ref{eq:J}) for the explicit formula.

Perhaps the most interesting is our third characterization 
which applies to any compact {\em connected} set $K$ in $\P^n$
(see Theorem \ref{thm:main2}). 
Let $S_K$ be defined by (\ref{eq:BK}). Then a point $x\in \P^n$ 
belongs to the projective hull $\wh K_{\P^n}$ if and only if there is a point 
$0\ne p\in L_x=\pi^{-1}(x)$ over $x$ and a sequence of analytic discs 
$F_j\colon \ol\D\to \C^{n+1}\setminus \{0\}$ satisfying 
\[
	F_j(0)=p \ (\forall j\in \N), \qquad 
	\lim_{j\to\infty} \max_{t\in[0,2\pi]} \dist(F_j(\E^{\I t}),S_K)=0.
\]
The projected sequence of analytic discs 
$f_j=\pi\circ F_j\colon \ol\D\to \P^n$ 
clearly enjoys the bounded lifting property and also 
\[
	f_j(0)=x\ (\forall j\in \N), \qquad 
	\lim_{j\to\infty} \max_{t\in[0,2\pi]} \dist(f_j(\E^{\I t}),K)=0.
\]
The nontrivial addition when compared to Theorem \ref{thm:main} 
is that the entire circle $\bT$ is mapped arbitrarily close 
to the set $K$ (resp.\ to $S_K$). 

Both the second and the third characterization mentioned above
are based on a result of L\'arusson and Sigurdsson 
\cite{Larusson-Sigurdsson2005} which expresses 
the Siciak-Zahar\-yuta extremal function of a connected
open set in $\C^n$ as the envelope of a certain disc
functional.

%%%%%%%%%%%%%%%%%%%%%%%%%%%%%%%%%%%%%%%%%%%%%%%%%%%%%%%%%%%%%%%%%%%%%
%                                                   								%
%																								    								%
%  PLURISUBHARMONIC HULLS ANALYTIC DISCS                            %
% 										     																					%
%																								    								%
%																								    								%
%%%%%%%%%%%%%%%%%%%%%%%%%%%%%%%%%%%%%%%%%%%%%%%%%%%%%%%%%%%%%%%%%%%%%
\section{Plurisubharmonic hulls and analytic discs}
\label{sec:Poletsky}
Recall that the plurisubharmonic hull of a compact set $K$ 
in a complex space $X$ is defined by
\[
	\wh K_{\Psh(X)} = \{x\in X\colon v(x)\le \sup_K v,\ 
	\forall v\in  \Psh(X)\}.
\]
If $X$ is a Stein space, then $\wh K_{\Psh(X)}$ coincides with 
the holomorphic hull $\wh K_{\cO(X)}$; in particular, 
$\wh K_{\Psh(\C^n)} = \wh K_{\cO(\C^n)}$ 
is the polynomial hull of $K$. 
The following result is due to Poletsky \cite{Poletsky1993} 
in the basic case when $X$ is a domain in $\C^n$, 
to Rosay \cite{Rosay1,Rosay2} 
when $X$ is a complex manifold, and to the authors 
\cite[Corollary 1.4]{DF-Indiana} in the general case stated here. 
(We give a somewhat more precise formulation that will be used
in the sequel. Additional references can be found in \cite{DF-Indiana}.)

\begin{theorem}
\label{thm:Poletsky}
Assume that $X$ is a locally irreducible complex space
and $\dist_X$ is a distance function on $X$ inducing the 
standard topology. Let $K$ be a compact set in $X$ 
whose plurisubharmonic hull $\wh K_{\Psh(X)}$ is compact. 
Then a point $x\in X$ belongs to $\wh K_{\Psh(X)}$ if and only if 
for every open relatively com\-pact set $\Omega\Subset X$ containing 
$\wh K_{\Psh(X)}$ there is a sequence of analytic discs 
$f_j\in \cA_\Omega$ satisfying $f_j(0)=x$ $(j=1,2,\ldots)$ and
\begin{equation}
\label{eq:Psequence}
	\big| \left\{ t\in[0,2\pi] \colon 
	\dist_X \left(f_j(\E^{\I t}),K\right)< 1/j  \right\} 
	\big|  	> 2\pi - 1/j,\quad j\in \N.
\end{equation}
Here $|\cdotp|$ denotes the Lebesgue measure on $\R$.
\end{theorem}

\begin{definition}
\label{def:P-seq}
A sequence of analytic discs $f_j$ satisfying the conditions
in the above theorem, with supports contained in a compact subset of $X$, 
is called a {\em P-sequence} for the pair $(K,x)$. 
\end{definition}

One can replace $1/j$ in (\ref{eq:Psequence}) 
by any sequence $\epsilon_j>0$ decreasing to $0$
without changing the conclusion of the theorem. 
The existence of a P-sequence for $(K,x)$  
trivially implies that $x\in \wh K_{\Psh(X)}$,
but the converse is nontrivial. 
Applying this theorem to a family of neighborhoods shrinking 
down to the hull $\wh K_{\Psh(X)}$ we obtain:

\begin{corollary}
\label{cor:1}
Assume that the sets $K\subset X$ satisfy the 
hypotheses of Theorem \ref{thm:Poletsky}.
For every point $x\in \wh K_{\Psh(X)}$ there exists 
a sequence of analytic discs $f_j\in \cA_X$ satisfying 
$f_j(0)=x$, the condition (\ref{eq:Psequence}), and also
\[
	\max_{\zeta \in \ol D} \dist_X \left(f_j(\zeta),\wh K_{\Psh(X)} \right)< 1/j, 
	\qquad 	j=1,2,\ldots.
\]
\end{corollary}

By Wold \cite{Wold2010}, Corollary \ref{cor:1} implies the 
following theorem of Duval and Sibony \cite{Duval-Sibony1} 
with additional control of the support of the current.

\begin{corollary}
\label{cor:Wold}
Let $K$ be a compact set in $\C^n$. 
For every point $p\in\wh K$ there exists 
a positive $(1,1)$-current T on $\C^n$ (acting on the space 
of $(1,1)$-forms with continuous coefficients) satisfying 
$p \in \supp T\subset \wh K$ and $\d\d^c T=\mu-\delta_p$, 
where $\mu$ is a probability measure on $K$ 
and $\delta_p$ is the Dirac mass at $p$.
\end{corollary}

The converse is trivial: If a current $T$ with these properties 
exists, then for any plurisubharmonic function 
$u\in\cC^2(\C^n)$ we have 
\[
   0\le T(\d\d^c u) = \int_K u\,\d\mu - u(p) \le \sup_K u - u(p)
\]
and hence $u(p)\le \sup_K u$.

In a Stein space $X$ we get the same description for the part
of the hull in the regular locus even if $X$ is not locally irreducible.
% The following result is essentially a corollary to 
% \cite[Theorem 1.1]{DF-Polonici}.

\begin{proposition}
\label{prop:2}
Let $K$ be a compact set in Stein space $X$ such that $K\not\subset X_{\rm sing}$.
Choose a relatively compact pseudoconvex Runge domain 
$V\Subset X$ containing $\wh K_{\cO(X)}$. 
Then a point $x\in X_{\rm reg}\cap V$ belongs to $\wh K_{\cO(X)}$ if and only
if for every open set $U\supset K$ and 
$\epsilon>0$ there exist a disc $f\in \cA_V$ 
and a set $E_f\subset [0,2\pi]$ of 
Lebesgue measure $|E_f| <\epsilon$ such that 
\[
	 	f(0)=x\quad {\rm and} \quad
	 	f(\E^{\I t})\in U\ \hbox{for all } t\in [0,2\pi]\bs E_f.  
\]
\end{proposition}

\begin{remark}
We do not know whether the same  conclusion
holds for points $x\in \wh K_{\cO(X)} \cap X_{\rm sing}$.
However, if $K$ is entirely contained in the singular locus $X_{\rm sing}$, 
then its holomorphic hull $\wh K_{\cO(X)}$ also lies in $X_{\rm sing}$ and 
$\wh K_{\cO(X)}=\wh K_{\cO(X_{\rm sing})}$, so we may apply 
Proposition \ref{prop:2} to the Stein space $X_{\rm sing}$.
\end{remark}

\begin{proof}
Assume first that a point $x\in X$ satisfies the stated conditions;
we shall prove that $x\in \wh K_{\cO(X)}$.
Choose a function $\rho\in \Psh(X)$. 
Set $M=\sup_K \rho$ and $M'=\sup_{V} \rho$. 
Pick a number $\epsilon>0$ and an open set $U$ with 
$K\subset U\Subset V$  such that $\sup_U \rho < M+\epsilon$. 
Let the disc $f\in \cA_V$ and the set $E_f\subset [0,2\pi]$ 
satisfy the hypotheses of the corollary. 
Recall that the Poisson functional $P_u(f)$ associated to 
an upper semicontinuous function $u\colon X\to\R\cup\{-\infty\}$ 
is defined by 
\begin{equation}
\label{eqn:Poisson}
	   P_u(f) = \int_0^{2\pi} u(f(\E^{\I t}))\, \frac{\d t}{2\pi},\qquad f\in \cA_X.
\end{equation}
Since $\rho$ is plurisubharmonic, we have  
\[
	\rho(x) \le P_{\rho}(f) = \int_{E_f} \rho(f(\E^{\I t}))\, \frac{\d t}{2\pi} 
	          + \int_{[0,2\pi]\bs E_f} \rho(f(\E^{\I t}))\, \frac{\d t}{2\pi} 
          < M'\epsilon + M+\epsilon.  
\]
Since this holds for every $\epsilon>0$, we get that
$\rho(x)\le M$. As $\rho\in \Psh(X)$ was arbitrary,
we conclude that  $x\in \wh K_{\Psh(X)}=\wh K_{\cO(X)}$.

Conversely, assume that $x\in \wh K_{\cO(X)} \cap X_{\rm reg}$. 
Since $V$ is a Runge pseudoconvex domain in a Stein space $X$, we have
$\wh K_{\cO(X)}=\wh K_{\cO(V)}=\wh K_{\Psh(V)}$.
The function $u\colon V\to [-1,0]$ which equals $-1$ on the 
open set $U\supset K$ and equals $0$ on $V \bs U$ 
is upper semicontinuous. Let $v\colon V\to\R$ be the envelope 
of the Poisson functional $P_u$ (\ref{eqn:Poisson}) corresponding to $u$.
Then clearly $-1\le v\le 0$ on $V$, and $v=-1$ on $U$.
According to \cite[Theorem 1.1]{DF-Polonici}, 
the function $v$ is plurisubharmonic on $V\cap X_{\rm reg}$.

The singular locus $X_{\rm sing}$ is closed and locally complete 
pluripolar. Since $X$ is Stein, it is also complete pluripolar 
(see \cite{Col,Dem}). Pick a plurisubharmonic function 
$\rho$ on $X$ such that $\rho^{-1}(-\infty)=X_{\rm sing}$.
The function $v+\epsilon \rho$ is then plurisubharmonic on $V$ 
for every $\epsilon>0$. Therefore
\[
   (v+\epsilon \rho)(x)\le \sup_K (v+\epsilon \rho) = -1 +\epsilon \sup_K \rho
\]
for every $\epsilon >0$. Since $x \in X_{\rm reg}$ and 
$K\not\subset X_{\rm sing}$, we get by letting $\epsilon\to 0$
that $v(x)=-1$. The definition of $v$ implies that for every
$\epsilon>0$ there is a disc $f\in \cA_V$ with $f(0)=x$ 
and $P_u(f)< -1+ \epsilon/2\pi$.
Hence the set $E_f=\{t\in [0,2\pi] \colon f(\E^{\I t}) \notin U\}$ has measure 
at most $\epsilon$.
\end{proof}

%%%%%%%%%%%%%%%%%%%%%%%%%%%%%%%%%%%%%%%%%%%%%%%%%%%%%%%%%%%%%%%%%%%%%
%                                                   								%
%																								    								%
%  CHARACTERIZATION OF PROJECTIVE HULLS BY ANALYTIC DISCS, 1     		%
% 										     																					%
%																								    								%
%																								    								%
%%%%%%%%%%%%%%%%%%%%%%%%%%%%%%%%%%%%%%%%%%%%%%%%%%%%%%%%%%%%%%%%%%%%%
%
\section{Sequences of analytic discs with the bounded lifting property}
\label{sec:proj-discs}
Let $\pi\colon \C^{n+1}_*=\C^{n+1}\setminus \{0\}\to \P^n$ be the 
standard projection; this is a holomorphic fiber bundle
with fiber $\C^*=\C\setminus\{0\}$, obtained by removing the
zero section from the universal line bundle $L\to\P^n$.
Every continuous map $f\colon \ol\D\to \P^n$ from the closed disc 
lifts to a continuous map $F\colon \ol\D\to \C^{n+1}_*$ 
as illustrated in the following diagram: 
\begin{equation}
\label{eq:diagram}
	\xymatrix{ 
	 & \C^{n+1}_*  \ar[d]^{\pi} \\ 
   \ol\D \ar[ur]^F \ar[r]^{f} & \P^n }
\end{equation}
The Oka principle shows that a holomorphic map $f$ 
can be lifted to a holomorphic map $F$ 
(see Corollary 5.4.11 in \cite[p.\ 196]{F-book}).
In fact, lifting $f$ is equivalent to finding 
a nowhere vanishing section, that is, a trivialization, 
of the pullback $f^*L$ of the universal bundle $L\to\P^n$. 
Since every holomorphic line bundle over the disc $\D$ 
(in fact, over any open Riemann surface) is  
holomorphically trivial, a lifting exists.

\begin{definition}
\label{def:BLP} 
A sequence of analytic discs $f_j\colon \ol\D\to \P^n$
enjoys the {\em bounded lifting property} if there
exist a constant $C>0$ and a sequence of analytic discs 
$F_j\colon\ol\D\to\C^{n+1}_*$ satisfying $\pi\circ F_j=f_j$ and  
\begin{equation}
\label{eq:BLP}
	\sup_{t\in[0,2\pi]} |F_j(\E^{\I t})| \le C\, |F_j(0)|,\qquad j=1,2,\ldots.
\end{equation}
\end{definition}

The following result describes the projective hull
in terms of Poletsky sequences of discs in $\P^n$ with the bounded lifting property. 
For the notion of a P-sequence see Def.\ \ref{def:P-seq} above.

\begin{theorem}
\label{thm:main}
Let $K$ be a compact set in $\P^n$. A point $x\in \P^n$
belongs to the projective hull $\wh K_{\P^n}$ if and only if 
there is a P-sequence $f_j\colon \ol\D\to\P^n$ 
for $(K,x)$ with the bounded lifting property.
\end{theorem}

\begin{remark}
The bounded lifting property in Theorem \ref{thm:main} is crucial.
Indeed, taking $K=\{p\}$ to be a singleton
and $x$ to be any point of $\P^n$, we consider a projective 
line $\P^1\cong \Lambda\subset\P^n$ through $x$ and $p$. By removing 
from $\Lambda$ a small disc around $p$ we obtain an analytic 
disc through $x$ that has all of its boundary as close as desired to $K$.
\qed\end{remark}

\begin{proof}
Fix a point $x\in \wh K_{\P^n}$. Let $L_x$ be the complex line through 
$0$ in $\C^{n+1}$ determined by $x$. By (\ref{eq:hull2}) 
the disc $\triangle_x = L_x\cap \wh S_K$ has positive radius $r(x)>0$.
Pick a point $p\in\triangle_x$ with $|p|=r(x)$. 
Theorem \ref{thm:Poletsky} furnishes a sequence of analytic discs 
\[
	F_j\colon \ol\D\to \B_{1+1/j}=\{z\in\C^{n+1}\colon |z|<1+1/j\}
\]
such that for all $j\in\N$ we have $F_j(0)=p$ and 
\begin{equation}
\label{eq:estimateF}
	\big| \left\{ t\in[0,2\pi] \colon 
	\dist_{\C^{n+1}} \left(F_j(\E^{\I t}),S_K\right) < 1/j  \right\} 
	\big|  	> 2\pi - 1/j.
\end{equation}
By a small deformation of $F_j$, keeping the centers $F_j(0)=p$ fixed,
we may assume that none of the image discs $F_j(\ol\D)$ contains the origin,
so $F_j(\ol\D) \subset\C^{n+1}_*$ for all $j\in\N$.
Set $f_j=\pi\circ F_j\colon \ol\D\to \P^n$; hence $f_j(0)=x$ for all $j$.
We endow the sphere $S\subset\C^{n+1}$ with the Riemannian metric
induced from $\C^{n+1}$, and $\P^n$ is endowed with the Fubini-Study metric.
In this pair of metrics the projection $\pi|_S\colon S\to \P^n$ 
has Lipschitz constant one; hence  $\pi$ has Lipschitz 
constant at most $2$ in some neighborhood of $S$. This implies that
\[
	\big| \left\{ t\in[0,2\pi] \colon 
	\dist_{\P^{n}} (f_j(\E^{\I t}),K) < 2/j  \right\} 
	\big|  	> 2\pi - 1/j.
\]
Thus the sequence of discs $f_{2j}$ in $\P^n$ is a P-sequence for 
the pair $(K,x)$. By the construction the sequence $f_j$ has 
the bounded lifting property; indeed, (\ref{eq:BLP}) holds for 
the constant $C=1/r(x) +\epsilon$ for any $\epsilon>0$ 
(but we can not take $\epsilon=0$). This establishes one of the implications.

To prove the converse implication, assume that there exists a P-sequence
$f_j\colon\ol\D\to \P^n$ for $(K,x)$ with the bounded lifting property.
Let $C>0$ be a constant satisfying (\ref{eq:BLP}). 
Pick a point $p\in \pi^{-1}(x)$ with $|p|=1/C$.
Let $F_j\colon \ol\D\to\C^{n+1}_*$ be a lifting of $f_j$
with $F_j(0)=p$ (this can be achieved by a rescaling). 
Then $\sup_{t\in[0,2\pi]} |F_j(\E^{\I t})| \le C|p|=1$ 
for all $j$. Inside the unit ball $\ol\B \subset\C^{n+1}$
the set $B_K$ (\ref{eq:BK}) is a union of fibers of $\pi$,
and the map $\pi$ is expanding in directions orthogonal to the 
complex lines $L_x$. Hence we have
\[
	\dist_{\C^{n+1}} \left(F_j(\E^{\I t}),B_K\right) \le 
	\dist_{\P^n} \left( f_j(\E^{\I t}),K\right)
\]
for all $t\in[0,2\pi]$ and $j\in\N$. 
Therefore the estimate (\ref{eq:estimateF}) holds, 
which means that $F_j$ is a P-sequence in $\C^{n+1}$ for 
the pair $(S_K,p)$. Hence the point $p$ belongs 
to the polynomial hull of $S_K$, and so the point 
$x=\pi(p)\in\P^n$  belongs to the projective hull of $K$.
\end{proof}

It would be interesting to understand conditions 
on a sequence of analytic discs in $\P^n$ implying the bounded lifting property.
Here is a simple observation: 
If $\Omega$ is an open simply connected Stein domain 
in $\P^n$ then the $\C^*$-bundle $\pi\colon \C^{n+1}\setminus \{0\} \to \P^n$
is holomorphically trivial over $\Omega$, and hence any sequence of discs
contained in a relatively compact subset of $\Omega$ has the bounded 
lifting property. This holds in particular if $\Omega$ is an affine 
chart $\C^n\subset \P^n$.

%%%%%%%%%%%%%%%%%%%%%%%%%%%%%%%%%%%%%%%%%%%%%%%%%%%%%%%%%%%%%%%%%%%%%
%                                                   								%
%																								    								%
%  HULLS OF CONNECTED AFFINE SETS                		                %
% 										     																					%
%																								    								%
%																								    								%
%%%%%%%%%%%%%%%%%%%%%%%%%%%%%%%%%%%%%%%%%%%%%%%%%%%%%%%%%%%%%%%%%%%%%
%
\section{Projective hulls of compact connected sets in affine charts}
\label{sec:HL}
In this section we extend a  theorem of Lawson and Wermer \cite{Lawson-Wermer},
which pertains to projective hulls of closed real curves in $\P^n$,
to an arbitrary compact connected set contained in an affine chart 
of $\P^n$. 

Assume that $\P^{n-1}\cong H \subset \P^n$ is a projective hyperplane 
and $\Omega$ is a nonempty open subset of $\C^n=\P^n\setminus H$.
Let $f\colon\ol\D\to \P^n$ be an analytic disc with $f(\bT)\subset \Omega$.
Then $f$ intersects $H$ in at most finitely many points of $\D$. 
The following quantity $J(f)$ (a disc functional) was first introduced 
by L\'arusson and Sigurdsson in \cite{Larusson-Sigurdsson2005}:
\begin{equation}
\label{eq:J}
	J(f) = -\sum_{\zeta\in\D} m_\zeta \log|\zeta|  \ge 0.
\end{equation}
Here $m_\zeta\in \Z_+$ denotes the intersection number of 
$f$ with $H$ at $\zeta$; so $m_\zeta=0$ if $f(\zeta) \notin H$.
The number $J(f)$ equals the value at the origin of the (positive) 
Green function on $\D$ that equals zero on $b\D$ and has logarithmic
poles at the finitely many points $\zeta_j\in \D$ for which $f(\zeta_j)\in H$
(see \cite[\S 4]{DF-Polonici}).  

The following result is due to Lawson and Wermer \cite[Theorem 1]{Lawson-Wermer}
for the case when $K$ is a connected closed curve. Our proof uses the same
ingredients, but is simpler due to a more systematic use
of inequalities involving the Siciak-Zaharyuta extremal functions.

\begin{theorem}
\label{thm:LW}
Assume that $K$ is a compact connected set in $\P^n$ and
$H\cong\P^{n-1}$ is a hyperplane in $\P^n$ such that $K\cap H=\emptyset$.
Then a point $p\in \C^n=\P^n\setminus H$ belongs to the 
projective hull of $K$ if and only if there exist a constant $0\le C<+\infty$ 
and a sequence of analytic discs $f_j\colon \ol\D\to\P^n$ 
$(j\in\N)$ satisfying the following three properties:
\begin{itemize}
\item[\rm (a)] 
	$f_j(0)=p$ for all $j=1,2,\ldots$,
\item[\rm (b)] 
	$\lim_{j\to\infty} \max_{t\in[0,2\pi]} \dist(f_j(\E^{\I t}),K)=0$, and
\item[\rm (c)] 
	$J(f_j)\le C$ for all $j=1,2,\ldots$.
\end{itemize}
If this holds then there exists a sequence of analytic discs $f_j$ 
with simple poles satisfying conditions (a), (b) and
\begin{itemize}
\item[\rm (c')]
	$\lim_{j\to\infty} J(f_j) = V_K(p)$, the value at $p$ of the 
	Siciak-Zaharyuta extremal function of the set $K\subset\C^n$.
\end{itemize} 
\end{theorem}

\begin{proof}
For every analytic disc $f\colon\ol\D\to\P^n$
with center $f(0)\in\C^n$ and boundary $f(\bT)$ contained in
an open set $\Omega\subset\C^n$ we have the inequality 
\[
	V_\Omega(f(0))\le J(f).
\]
(See the  proof of Theorem 4.2 in \cite{DF-Polonici}. 
The point is simply that $V_\Omega\circ f$ is a subharmonic function 
on $\ol\D\setminus f^{-1}(H)$ that vanishes on $\bT$ and has 
a logarithmic pole with weight $m_j$ at every point $\zeta_j$ of the divisor
$f^{-1}(H)=\sum m_j \zeta_j$. Hence it is bounded above by the Green function 
on $\D$ with the same poles. Comparing the values at the origin 
gives the stated inequality.)

Assume now that $\Omega_1\supset \Omega_2\supset \cdots$ are 
open sets with $\cap_{j=1}^\infty \Omega_j=K$, and $f_j\colon\ol\D\to\P^n$
is a sequence of analytic discs with $f_j(0)=p$, $f_j(\bT)\subset \Omega_j$,
and $J(f_j)\le C<\infty$ for all $j\in\N$. Then
$V_{\Omega_j}(p) \le J(f_j)\le C$ for all $j$. As $j\to\infty$, the numbers
$V_{\Omega_j}(p)$ increase to $V_K(p)$, so we get $V_K(p)\le C<+\infty$.
Thus $p$ belongs to the projective hull of $K$.

Conversely, assume that $p\in \wh K_{\P^n}$; hence $V_K(p)<+\infty$.
Since $K$ is connected, we can choose a decreasing sequence 
of connected open neighborhoods $\Omega_j\supset K$
as above. Pick a decreasing sequence of numbers $\epsilon_j>0$ converging to zero.
By L\'arusson and Sigurdsson \cite{Larusson-Sigurdsson2005}
we have for every connected open set $\Omega\subset\C^n \subset\P^n$ that
\[
		V_\Omega(p) = \inf_{f} J(f),
\]
the infimum being taken over all discs in $\P^n$ with $f(0)=p$  
and $f(\bT)\subset\Omega$. Hence there exists  for every $j\in \N$ 
an analytic disc $f_j\colon\ol\D\to\P^n$ 
such that $f_j(0)=p$, $f_j(\bT)\subset \Omega_j$, and 
\[
	V_{\Omega_j}(p) \le J(f_j)< V_{\Omega_j}(p)+\epsilon_j.
\]
By the transversality theorem we may assume that each $f_j$ 
has simple poles, i.e., it intersects the hyperplane $H$ transversely.
As $j\to\infty$, the numbers $V_{\Omega_j}(p)$ increase monotonically to 
$V_K(p)$. It follows that the sequence $f_j$ 
satisfies properties (a), (b) and (c').
\end{proof}

\begin{remark}
At this point one can proceed as in \cite{Lawson-Wermer} to 
write $f_j=G_j/B_j$, where $G_j\colon\ol\D\to\C^n$ is a
holomorphic disc in $\C^n$ and $B_j$ is a 
finite Blaschke product whose zeros are precisely the poles
of $f_j$ (i.e., the points in $f^{-1}(H)$). The condition
$J(f_j)\le C$ implies that $|B_j(0)|\ge \E^{-C}>0$ for all $j$. 
Furthermore, on $\bT$ we have $|G_j|=|f_jB_j|=|f_j|$ 
which is uniformly bounded, and hence the sequence 
$|G_j|$ is uniformly bounded on $\ol\D$ by the maximum principle.
Passing to subsequences we may assume that the sequence 
$B_j$ converges uniformly on 
compacts in $\D$ to a nonzero Blaschke product $B$, 
and the sequence $G_j$ converges to a bounded holomorphic map $G\colon\D\to\C^n$. 
Lawson and Wermer then show that the holomorphic map 
$f=G/B\colon \D\to\P^n$ satisfies $f(\D)\subset \wh K_{\P^n}$.
However, since nothing in this argument 
prevents $f$ from being the constant map $f\equiv f(0)$,
the limit discs obtained in this way do not seem to give a 
satisfactory description of the projective hull. 
\qed \end{remark}

%%%%%%%%%%%%%%%%%%%%%%%%%%%%%%%%%%%%%%%%%%%%%%%%%%%%%%%%%%%%%%%%%%%%%
%                                                   								%
%																								    								%
%  CHARACTERIZATION OF PROJECTIVE HULLS BY ANALYTIC DISCS        		%
% 										     																					%
%																								    								%
%																								    								%
%%%%%%%%%%%%%%%%%%%%%%%%%%%%%%%%%%%%%%%%%%%%%%%%%%%%%%%%%%%%%%%%%%%%%
%
\section{Characterization of the projective hull of a compact 
connected set by analytic discs} 
\label{sec:universal}
In this section we obtain the following result which improves 
Theorem \ref{thm:main} in the case when the compact set 
$K\subset \P^n$ is also connected.

\begin{theorem}
\label{thm:main2}
Let $K$ be a compact connected set in $\P^n$.
A point $p\in \C^{n+1}\setminus \{0\}$ belongs to the polynomial
hull of the set $S_K \subset \C^{n+1}$  (\ref{eq:BK}), and hence 
$x=\pi(p) \in\P^n$ belongs to the projective hull of $K$,
if and only if there exists a sequence of analytic discs 
$F_j\colon \ol\D\to \C^{n+1}\setminus \{0\}$ such that 
\begin{equation}
\label{eq:Fj}
	F_j(0)=p \ (\forall j\in \N), \qquad 
	\lim_{j\to\infty} \max_{t\in[0,2\pi]} \dist(F_j(\E^{\I t}),S_K)=0.
\end{equation}
\end{theorem}

\begin{remark}
By the maximum principle the images $F_j(\ol\D)$
are contained in balls $(1+\epsilon_j) \B$ with 
$\epsilon_j= \max_{t\in[0,2\pi]} \dist(F_j(\E^{\I t}),S_K)$.
However, it does not seem possible to control from below the 
distance of $F_j(\ol\D)$ to the origin. 
The projected sequence $f_j=\pi\circ F_j \colon \ol\D\to\P^n$
then clearly enjoys the bounded lifting property 
(see Def.\ \ref{def:BLP}) and satisfies 
\[
	f_j(0)=x \ (\forall j\in \N), \qquad 
	\lim_{j\to\infty} \max_{t\in[0,2\pi]} \dist(f_j(\E^{\I t}),K)=0.
\]
The advantage over the P-sequence found in Theorem \ref{thm:main} is
that the entire boundary circle $f_j(\bT)$
is mapped close to the set $K$, similarly to what happened in
Theorem \ref{thm:LW}. However, the set $K$ in Theorem \ref{thm:main2}
need not be contained in any affine chart of $\P^n$.
\qed\end{remark}

\begin{proof}
The existence of a sequence of discs satisfying condition (\ref{eq:Fj}) 
clearly implies that the point $p$ belongs to the polynomial hull of $S_K$.

Assume now that a  point $p\in \C^{n+1}\setminus \{0\}$ 
belongs to $\wh{S_K}$; we shall find a sequence $F_j$
satisfying (\ref{eq:Fj}). To this end we 
consider analytic discs as in the diagram (\ref{eq:diagram}). 
We compactify $\C^{n+1}$ by adding the hyperplane at infinity and obtain 
$\P^{n+1} = \C^{n+1} \cup H_\infty$. 
Given a set $\Omega\subset\C^{n+1}$
we recall that $V_\Omega$ is the Siciak-Zaharyuta extremal function with
logarithmic pole at $H_\infty$. Since the point $p$
belongs to the polynomial hull $\wh S_K$, we have $V_{S_K}(p)=0$.
Choose a decreasing sequence of connected
open sets $\Omega_j \subset \C^{n+1}$,  
\[
	\Omega_1\supset\Omega_2\supset \cdots \supset \cap_{j=1}^\infty \Omega_j=S_K,
\]
such that every $\Omega_j$ is circular
(invariant with respect to the circle action $(t,z) \mapsto \E^{\I t}z$). 
For every $j$ we pick a smaller circular neighborhood 
$\Omega'_j$ of the set $S_K$ and a number $\epsilon_j>0$ 
such that $\E^{\epsilon_j} z  \in \Omega_j$ for every $z\in \Omega'_j$.

Since $V_{S_K}(p)=0$, we have $V_{\Omega'_j}(p)=0$ for all $j$. 
By L\'arusson and Sigurdsson 
\cite{Larusson-Sigurdsson2005} there exists for every $j$ 
an analytic disc $G_j\colon\ol\D\to\P^{n+1}$ such that 
$G_j(0)=p$, $G_j(\bT)\subset \Omega'_j$, and 
$J(G_j) < \epsilon_j$. Here $J(G_j)=-\sum_k m_{j,k} \log|\zeta_{j,k}|$,  
where $\sum_k m_{j,k} \zeta_{j,k} = G_j^{-1}(H_\infty)$
is the intersection divisor of $G_j$  with the 
hyperplane $H_\infty$. By general position we may assume that $G_j(\ol\D)$ 
does not contain the origin $0\in\C^{n+1}$ for any $j$.

Let $B_j(\zeta)$ denote the Blaschke product with the 
zeros $\zeta_{j,k}$ of multiplicity $m_{j,k}$. 
A calculation gives  
\[
	|B_j(0)| = \E^{-J(G_j)} > \E^{-\epsilon_j}.
\] 
Define new analytic discs by 
\[
	F_j(\zeta) = \frac{B_j(\zeta)}{B_j(0)} \, G_j(\zeta),
	\qquad  j=1,2,\ldots.
\]
Then $F_j$ is an analytic disc in $\C^{n+1}\setminus\{0\}$ 
(since the poles of $G_j$ are exactly cancelled off by the zeros 
of $B_j$ and no additional zeros appear), 
and $F_j(0)=G_j(0)=p$. Since $|B_j|=1$ on $\bT$ and the sets 
$\Omega'_j\subset \Omega_j$ are circular, our choice of the number
$\epsilon_j$ implies that $F_j(\bT) \subset \Omega_j$.
Hence the sequence $F_j$ satisfies the stated properties.
\end{proof}

The above proof does not use any special hypothesis
of the set $S_K$ other that it is connected and circular.
Hence we get the following result of possible independent interest.
It vaguely resembles the description of the polynomial 
hulls of sets in $\C^2$ fibered over the unit circle, 
with disc fibers, due to Alexander and Wermer \cite{AW}.

%
%
%  Polynomial hull of a compact connected circular set
%
%
\begin{theorem}
\label{thm:hull-circular}
Assume that $K$ is a compact connected set in $\C^n$ which is
invariant with respect to the circle action $(t,z)\mapsto \E^{\I t}z$.
Then a point $p\in \C^n$ belongs to the polynomial
hull of $K$ if and only if there exists a sequence of analytic discs 
$f_j\colon \ol\D\to \C^{n}$ such that 
\[	f_j(0)=p \ (\forall j\in \N), \qquad 
	\lim_{j\to\infty} \max_{t\in[0,2\pi]} \dist(f_j(\E^{\I t}),K)=0.
\]
\end{theorem}

Theorem \ref{thm:hull-circular} fails in general for a disconnected circular set $K$. 
A simple example for which the conclusion fails is the union $K=T_1\cup T_2$ 
of two disjoint totally real tori $T_1,T_2$ in the unit sphere of $\C^2$
such that $K$ bounds an embedded complex annulus $A\subset \C^2$,
with the two boundary circles of $A$ contained in different connected 
components of $K$. However, we do not know the answer to the following
question.

\begin{problem} Does Theorem \ref{thm:hull-circular} still hold
if the set $K$ is not circular? 
\end{problem}

\subsection*{Acknowledgement}
We wish to thank Norman Levenberg
and Finnur L\'aru\-sson for their remarks
on an earlier version of the paper. The question whether
the projective hull can be characterized by sequences of 
analytic discs was communicated to us by Levenberg.

%%%%%%%%%%%%%%%%%%%%%%%%%%%%%%%%%
%                     					%
%																%
%  THE BIBLIOGRAPHY         		%
% 															%
%																%
%																%
%%%%%%%%%%%%%%%%%%%%%%%%%%%%%%%%%

\end{document}